\documentclass[twoside,11pt]{article}%
\usepackage{amssymb}
\usepackage[scr=rsfs]{mathalpha}
\usepackage{bm}
\usepackage{latexsym}
\usepackage{amsmath,amssymb,epsfig,subfigure}
\usepackage{amsthm,enumerate,verbatim}
\usepackage{amsfonts}
\usepackage{color}
\usepackage{amsmath}
\usepackage{graphicx}
\usepackage{mathrsfs}
\usepackage{float}
\usepackage[numbers,sort&compress]{natbib}
\providecommand{\U}[1]{\protect\rule{.1in}{.1in}}
\providecommand{\U}[1]{\protect\rule{.1in}{.1in}}
\newcommand{\diag}{{\rm diag}}

\newcommand{\BE}{\begin{equation}}
	\newcommand{\EE}{\end{equation}}

\numberwithin{equation}{section}
\newtheorem{proposition}{Proposition}[section]
\newtheorem{theorem}[proposition]{Theorem}
\newtheorem{lemma}[proposition]{Lemma}
\newtheorem{corollary}[proposition]{Corollary}
\newtheorem{definition}[proposition]{Definition}
\newtheorem{remark}[proposition]{Remark}

\allowdisplaybreaks

\begin{document}
	%
	%
	%
	%
	%
	
	\title{{\bf{Further analysis of weighted integral inequalities for improved exponential stability analysis of time delay neural networks systems}}}
		\author{	Yuanyuan Zhang$^a$ \quad Han Xue$^a$\quad  Kachong Lao$^a$\quad Chonkit Chan$^a$    \\Chenyang Shi$^{b}$\thanks{ Corresponding author. Email: cyshi@scau.edu.cn. }
		\quad
		Seakweng Vong$^a$\\
		{\footnotesize  \textit{a.Department of Mathematics, University of Macau, Avenida da Universidade, Macau, China.}}\\
		{\footnotesize\textit{	b.Department of Mathematics, South China Agricultural University, Guangzhou, China.}}
	}
	\date{}
	\maketitle	
	
	\begin{abstract}
		This work investigates the exponential stability of neural networks (NNs) systems with time delays. By considering orthogonal polynomials with weighted terms, a new weighted integral inequality is presented. This inequality extend several recently established results. Additionally, based on the reciprocally convex inequality, this study focuses on analyzing the exponential stability of systems with time-varying delays that include an exponential decay factor, a weighted version of the reciprocally convex inequality is first derived. Utilizing these inequalities and the suitable Lyapunov-Krasovskii functionals (LKFs) within the framework of linear matrix inequalities (LMIs), the new criteria for the exponential stability of NNs system is obtained. The effectiveness of the proposed method is demonstrated through multiple numerical examples.
	\end{abstract}

	\section{Introduction}
	NNs systems play a crucial role in modern technology and engineering, particularly in the analysis of complex data in applications such as image recognition, scheduling, and resource allocation \cite{S1,WM,ChenW}. One of the main challenges faced by neural network systems is time delay, which includes physical transmission delays, computational delays, and delays caused by inter-node communication. When the system's delay becomes excessive or exhibits significant fluctuations, it can lead to oscillations, instability, and other abnormal behaviors. As delays increase, the prediction accuracy and decision quality of the neural network may decline, resulting in poor performance during task execution \cite{HeY1,HuaC,HBZ}. Therefore, the ability to effectively address delays is a key factor in assessing the strengths and weaknesses of the system. Consequently, researchers are focused on optimizing the stability of neural network systems through a comprehensive analysis of delays, leading to the derivation of numerous stability results based on linear matrix inequalities.
	
	The exponential stability analysis of time-delay NNs primarily utilizes the Lyapunov-Krasovskii functional (LKF) method to derive the achievable delay upper bound (ADUB) \cite{ZYL}. The objective is to ensure that all delays within this ADUB maintain the stability of the NNs system. The approach can be divided into two key components. First, a LKF is established to express the stability of the system, as described in the literature. Next, appropriate integral inequalities are employed to estimate the derivative of the LKF. Enhanced LKFs have been proposed in some studies, taking into account additional delay terms. Generally, the inclusion of more system information in the construction of the LKF tends to reduce conservatism. Using the LKF method, certain integral inequalities are applied to estimate the integral terms derived from the LKF derivatives. The challenge lies in obtaining tighter bounds on these derivatives \cite{OM}. Consequently, constructing suitable LKFs and refining inequalities to estimate these integral terms has become a focal point of research \cite{Zeng,Levan1}.

	In general, Jensen inequality is used to bound the cross terms in the derivative of LKFs, the most studies will improve Jensen's inequality to obtain less conservative \cite{Hien2}. In recent studies, the Wirtinger-based inequality has been applied to deal with integral terms, which is a special case of Jensen's inequality \cite{As}.  For some complex orthogonal functions, Bessel-Legendre inequality is extended which can better estimate the integral terms  \cite{Seuret}. Some articles have also extended the analysis based on auxiliary inequalities, resulting in stricter outcomes than those derived from Jensen's inequality. In this paper, we similarly investigate the use of auxiliary inequalities to develop less conservative stability criteria. In addition, for handling convex optimization integral item, some research combined Jensen inequality and reciprocally convex inequality, which can be applied reciprocals of functions problem, it need too many variables in system \cite{Park,Vong}. Therefore, to obtain less conservative exponential stability criteria, we focus on new extensions of auxiliary function integral inequalities, while incorporating weighted terms with RCI.
	
	In the current study, our objective is to investigate the exponential stability of neural networks with time-varying delays, following the ideas presented in \cite{LiuX,SCY}. Based on the aforementioned analysis and research motivation, we extend the weighted inequalities discussed in \cite{Vong} by applying orthogonal cubic polynomials to the exponential stability analysis of neural network systems. Additionally, to address the convex terms with time delays, we combine weighted terms with the RCI for the first time, establishing a new integral inequality. This inequality requires fewer estimation terms, making it computationally less complex while still producing equally stringent bounds. The new LKF represents an improvement over the LKF in \cite{SCY}, incorporating new integral terms that include more delay information. By applying the newly derived weighted integral inequalities, weighted reciprocally convex inequality (WRCI), and the improved LKF, we obtain less conservative exponential stability criteria for neural network systems. Finally, the effectiveness of the new method is validated through several numerical examples.

	The following are main contributions of this paper:
	\begin{itemize} 
		\item[1.]To facilitate the analysis of time delays with exponential decay factors, we introduce a weighted reciprocally convex inequality for the first time. Additionally, we present a weighted integral inequality to effectively manage the integral cross terms associated with exponential decay factors. These inequalities enable the derivation of less conservative stability criteria for time-varying delay systems while also reducing the number of decision variables.
		\item [2.]By integrating the newly derived inequalities, we introduce a novel LKF that enhances the application of the WRCI while incorporating additional system information. As a result, new stability conditions are established through the use of the WRCI and other inequalities. The effectiveness of the proposed approach is validated through several numerical examples.
	\end{itemize}
	
	The following is the general framework of this paper. In the next section, two new generalized integral inequalities are given, namely improved auxiliary inequalities and weighted reciprocally convex inequality. In the following section, we show the application of two new inequalities to LKFs and give the main theoretical results. In the last section, numerical results are presented to validate the new stability criterion.
%
%
%
%
%
%
%

\bigskip
{\bf Notations:}
In this work, we denote \( \mathbb{R}^n \) as the n-dimensional Euclidean space. We use \( \mathbb{R}^{n \times m} \) to represent the set of all \( n \times m \) real matrices, \( \mathbb{S}^n \) for the collection of \( n \times n \) symmetric matrices, and \( \mathbb{S}_+^n \) for the set of \( n \times n \) symmetric positive definite matrices. Additionally, let \( * \) denote symmetric elements within a matrix. Finally, we define the notation \( sym(A) = A + A^T \), where \( T \) indicates the transpose of a matrix.

\section{Preliminaries}
We study  neural networks system with time-varying delay described by the following equation:
\begin{equation} \label{a1}
	\dot{z}(t) = -K_0z(t) + K_1f(z(t)) + K_2f(z(t-h(t))) + \epsilon,
\end{equation}
where the neuron state vector is represented as \( z(\cdot) = [z_{1}(\cdot), z_{2}(\cdot), \ldots, z_{n}(\cdot)]^{T} \in \mathbb{R}^{n} \) and  \( f(z(\cdot)) = [f_{1}(z_{1}(\cdot)), f_{2}(z_{2}(\cdot)), \ldots, f_{n}(z_{n}(\cdot))]^{T} \in \mathbb{R}^{n} \) be
 the activation function. The vector \( \epsilon = [\epsilon_1, \epsilon_2, \ldots, \epsilon_{n}]^{T} \in \mathbb{R}^{n} \) serves as the constant input to the network.

The diagonal matrix \( K_0 = \text{diag}\{k_{1}, k_{2}, \ldots, k_{n}\} \) has entries \( k_{i} > 0 \). The matrices \( K_1 \) and \( K_2 \) represent the interconnection weight matrices. The differentiable function \( h(t) \) indicates the time-varying delay, satisfying the condition as follow:
\begin{equation} \label{a2}
	0\leq h(t)\leq h
\end{equation}
and
\begin{equation} \label{a3}
	{|\dot{h}(t)|}\leq \mu
\end{equation}
for certain constants \( \mu \) and \( h \). Consistent with previous research, we assume that each activation function in (\ref{a1}) meets the following criteria:
\begin{equation} \label{a4}
	0\leq \frac{f_{j}(s)-f_{j}(v)}{s-v}\leq \mathscr{L}_{j},\quad s,v\in \mathbb{R},\quad s\neq v,\quad j=1,2,\ldots,n,
\end{equation}
for some positive constants $\mathscr{L}_{j},~j=1,2,\ldots,n$.\\
\indent Under the condition (\ref{a4}), following equation holds for $z^{*}=[z^{*}_{1},z^{*}_{2},\ldots,z^{*}_{n}]^{T}$ 
\begin{equation} \label{a5}
	K_0z^{*}=K_1f(z^{*})+K_2f(z^{*})+\epsilon.
\end{equation}
\indent
Next, we shift the equilibrium point \( s^{*} \) of system (\ref{a1}) to the origin using the transformation \( r(\cdot) = z(\cdot) - z^{*} \). Thus, the vector \( r = [r_{1}(\cdot), r_{2}(\cdot), \ldots, r_{n}(\cdot)]^{T} \) satisfies
\begin{equation} \label{a6}
	\dot{r}(t)=-K_0r(t)+K_1g(r(t))+K_2g(r(t-h(t)))
\end{equation}
where $g(r(\cdot))=[g_{1}(r_{1}(\cdot)),g_{2}(r_{2}(\cdot)),\ldots,g_{n}(r_{n}(\cdot))]^{T}$
and $g_{j}(r_{j}(\cdot))=f_{j}(r_{j}(\cdot)+z^{*}_{j})-f_{j}(z_{j}^{*}), j=1,2,\ldots,n$.
Then we can get
\begin{equation} \label{a7}
	0\leq \frac{g_j(r_{j})}{r_{j}}\leq \mathscr{L}_{j},\quad g_{j}(0)=0,\quad \forall r_{j}\neq 0,\quad j=1,2,\ldots,n.
\end{equation}
The exponential stability study of system  (\ref{a6})  is given below.
\begin{definition} \label{Definition1}
For \( t > 0 \), \( E \geq 1 \), the convergence rate \( k > 0 \) , the NNs system (\ref{a6}) is considered exponentially stable at the origin, the following inequality holds:
\[
\|r(t)\| \leq E\phi e^{-kt}
\]
 where \( \phi = \sup_{-h \leq \theta \leq 0} \|r(\theta)\| \).
\end{definition}
Before presenting the main result, we will first review some key lemmas.

			\begin{lemma}\cite{Vong} \label{L1}
			Let matrix $\Gamma \in  \mathbb{S}_+^n$ and any function $\varrho$ : $[c_1,c_2]\rightarrow \mathbb{R}^n$, scalar $c_2>c_1$,$\delta>0$, $g_k$ are polynomial functions, the following inequality hold:
			\begin{align}
				\int_{c_1}^{c_2} e^{\delta(v-c_2)} \varrho^T(t) \Gamma \varrho(t) d v \geq \sum_{k=0}^2 \frac{1}{\left\langle g_k, g_k\right\rangle_w}\left(\int_{c_1}^{c_2}  \varrho({t}) g_k d v\right)^T \Gamma\left(\int_{c_1}^{c_2}  \varrho({t}) g_k d v\right)
			\end{align}
			where ${\langle g_k,g_k\rangle}_w=\int_{c_1}^{c_2}  e^{-\delta(v-c_2)}g_k(v)g_k(v)dv$, k=0,1,2,
			$\varLambda_{i}={\langle v^i,1\rangle}_w,\ i=0,1,2,3,4;$ $\bar{k}=-\varLambda_{1}/\varLambda_{0},m=(\varLambda_{2}^2-\varLambda_{1}\varLambda_{3})/(\varLambda_{1}^2-\varLambda_{2}\varLambda_{0}),c=(\varLambda_0\varLambda_{3}-\varLambda_{1}\varLambda_{2})/(\varLambda_{1}^2-\varLambda_{2}\varLambda_{0})$ and ${\langle g_{0},g_{0}\rangle}_w=\varLambda_0,{\langle g_{1},g_{1}\rangle}_w=\varLambda_{2}+2\bar{k}{\varLambda_{1}}+\bar{k}^2\varLambda_{0},{\langle g_{2},g_{2}\rangle}_w=\varLambda_{4}+2c\varLambda_{3}+(c^2+2m)\varLambda_{2}+2mc\varLambda_{1}+m^2\varLambda_{0}$.
		\end{lemma}
		Inspired by the derivation process of weighted inequalities in \cite{Vong}, we have the following lemmas:
		\begin{lemma}  \label{L2}
			Let matrix $\Gamma\in  \mathbb{S}_+^n$  and any function $\varrho$ : $[c_1,c_2]\rightarrow \mathbb{R}^n$, scalar $c_2>c_1$,$\delta>0$, for any nonnegative integer $m$, $g_k$ are polynomial functions, then the following lemma hold:
			\begin{align}
				\int_{c_1}^{c_2} e^{\delta(v-c_2)} \varrho^T(v) \Gamma \varrho(v) d v \geq \sum_{k=0}^{\hat{m}} \frac{1}{\left\langle g_k, g_k\right\rangle_w}\left(\int_{c_1}^{c_2}  \varrho({v}) g_k d v\right)^T \Gamma\left(\int_{c_1}^{c_2}  \varrho({v}) g_k d v\right)
			\end{align}
		\end{lemma}
		For simplicity, we take $\hat{m}=3$,  then	$\varLambda_{i}={\langle v^i,1\rangle}_w,\ i=0,\ldots,5;$ $\hbar=(-\varLambda_{5}\varLambda_{1}^{2}+\varLambda_{4}\varLambda_{1}\varLambda_{2}+\varLambda_{1}\varLambda_{3}^{2}-\varLambda_{2}^{2}\varLambda_{3}+\varLambda_{0}\varLambda_{5}\varLambda_{2}-\varLambda_{0}\varLambda_{4}\varLambda_{3})/(\varLambda_{4}\varLambda_{1}^{2}-2\varLambda_{1}\varLambda_{2}\varLambda_{3}+\varLambda_{2}^{3}-\varLambda_{0}\varLambda_{2}\varLambda_{4}+\varLambda_{0}\varLambda_{3}^{2})$, $	q=(-I^{2}_{2}\varLambda_{4}+\varLambda_{2}\varLambda_{3}^{2}+\varLambda_{1}\varLambda_{5}\varLambda_{2}-\varLambda_{1}\varLambda_{3}\varLambda_{4}-\varLambda_{0}\varLambda_{5}\varLambda_{3}+\varLambda_{0}\varLambda_{4}^{2})/(\varLambda_{4}\varLambda_{1}^{2}-2\varLambda_{1}\varLambda_{2}\varLambda_{3}+\varLambda_{2}^{3}-\varLambda_{0}\varLambda_{2}\varLambda_{4}+\varLambda_{0}\varLambda_{3}^{2})$, $r=-(\varLambda_{5}\varLambda_{2}^{2}-2\varLambda_{2}\varLambda_{3}\varLambda_{4}+\varLambda_{3}^{3}-\varLambda_{1}\varLambda_{5}\varLambda_{3}+\varLambda_{1}\varLambda_{4}^{2})/(\varLambda_{4}\varLambda_{1}^{2}-2\varLambda_{1}\varLambda_{2}\varLambda_{3}+\varLambda_{2}^{3}-\varLambda_{0}\varLambda_{2}\varLambda_{4}+\varLambda_{0}\varLambda_{3}^{2})$, $\left\langle g_{3},g_{3}\right\rangle_{w}=\varLambda_{6}+2\hbar \varLambda_{5}+(\hbar^{2}+2q)\varLambda_{4}+(2\hbar q+2r)\varLambda_{3}+(2\hbar r+q^{2})I_{2}+2qr\varLambda_{1}+r^{2}\varLambda_{0}.$
	 Let $\varrho(t)=\dot{x}(t)$, we can obtain the following corollary:
		\begin{corollary}\label{C4}
			%
			{\small 
				\begin{align}
					\int_{c_1}^{c_2}  e^{\delta(v-c_2)} \dot{x}^T  \Gamma \dot{x} d v \geq \frac{1}{\left\langle g_0, g_0\right\rangle_w} \varOmega_0^T  \Gamma \varOmega_0+\frac{1}{\left\langle g_1, g_1\right\rangle_w} \varOmega_1^T  \Gamma \varOmega_1+\frac{1}{\left\langle g_2, g_2\right\rangle_w} \varOmega_2^T  \Gamma \varOmega_2+\frac{1}{\left\langle g_3, g_3\right\rangle_w} \varOmega_3^T  \Gamma \varOmega_3
				\end{align}
				Where,
				\begin{align}\notag
					\varOmega_0= & x(c_2)-x(c_1) \\\notag
					\varOmega_1= & (k+c_2) x(c_2)-(k+c_1) x(c_1)-\int_{c_1}^{c_2} r(v) d v \\\notag
					\varOmega_2= & \left(c_2^2+c c_2+m\right) x(c_2)-\left(c_1^2+c c_1+m\right) x(c_1)-(c+2 c_1) \int_{c_1}^{c_2} r(v) d v-2 \int_{c_1}^{c_2} \int_v^{c_2} x(u) d u d v \\\notag
					\varOmega_3= & \left(c_2^3+\hbar c_2^2+q c_2+r\right) x(c_2)-\left(c_1^3+\hbar c_1^2+q c_1+r\right) x(c_1)-\left(3 c_1^2+2 c_1 \hbar+q\right) \int_{c_1}^{c_2} r(v) d v \\\notag
					& -(2 \hbar+6 c_1) \int_{c_1}^{c_2} \int_v^{c_2} x(u) d u d v-6 \int_{c_1}^{c_1} \int_v^{c_2} \int_u^{c_2} x(s) d s d u d v\\\notag
				\end{align}
			}
		\end{corollary}
		\begin{remark}\label{R6}
			In particular cases, when $\delta$ approaches zero, the coefficients of the orthogonal functon are
			\begin{align}\notag
				k=-\frac{c_2+c_1}{2}&,&m&=\frac{a^{2}+4c_1c_2+c_2^{2}}{6},&c=&-c_2-c_1,\\\notag
				\hbar=-\frac{3(c_1+c_2)}{2}&,&q&=\frac{3(c_1^{2}+3c_1c_2+c_2^{2})}{5},&r=&-\frac{(c_1+c_2)(c_1^2+8c_1c_2+c_2^{2})}{20},
			\end{align}
			and the inequalities (2.2) will become the following inequalities,
			\begin{align}\label{e10}
				\int^{c_2}_{c_1}\varrho^{T}(v) \Gamma\varrho(v)dv\geq \frac{1}{c_2-c_1}\zeta^{T}\bar{ \Gamma}\zeta
			\end{align}
			where 
			{\footnotesize\begin{align}\notag
					&\zeta=\Biggl(\int^{c_2}_{c_1}\varrho(v)dv,\int^{c_2}_{c_1}(k+v)\varrho(v)dv,\int^{c_2}_{c_1}(v^{2}+mv+c)\varrho(v)dv,\int^{c_2}_{c_1}(v^{3}+\hbar v^{2}+qv+r)\varrho(v)dv\Biggl),\\\notag
					&\bar{ \Gamma}=diag\biggl\{ \Gamma,\frac{12}{(c_2-c_1)^{2}} \Gamma,\frac{180}{(c_2-c_1)^{4}} \Gamma,\frac{2800}{(c_2-c_1)^{6}} \Gamma\biggl\}.
			\end{align}}
			When $\delta $ approaches zero, the inequality (2.8)  becomes four terms Wirtinger-based inequality in\cite{Seuret}. Based on previous work, we take a cubic polynomial then can get above inequality. Obviously, considering the higher polynomial, getting the less conservation in result. However, we need to consider the computation of  Wirtinger-based inequality, and the inequality with weighted terms is difficult to directly apply   Wirtinger-based inequality to deal with the stability problem of time-delay systems.
		\end{remark}
		\begin{lemma}\label{L7}
			\cite{Park}	For given matrices $ \Gamma \in \mathbb{S}_+^{n}$, scalar $W_i\in \mathbb{R}^n (i=1,2)$, a function $\Theta$, if there exists a matrix $S \in \mathbb{R}^{n \times n}$ such that $\left[\begin{array}{ll} \Gamma & S \\ * &  \Gamma\end{array}\right] \geq 0$ then the inequality \begin{align}
				\Theta=\left[\begin{array}{l}W_1 \\ W_2\end{array}\right]^{T}
				\left[\begin{array}{cc}\frac{1}{\bar{\sigma}}  \Gamma & 0 \\ 0 & \frac{1}{1-\bar{\sigma}}  \Gamma\end{array}\right]\left[\begin{array}{l}W_1 \\ W_2\end{array}\right]
				\geq\left[\begin{array}{l}W_1 \\ W_2\end{array}\right]^{T}\left[\begin{array}{cc} \Gamma & S \\ * &  \Gamma\end{array}\right]\left[\begin{array}{l}W_1 \\ W_2\end{array}\right]
			\end{align} holds for all $\bar{\sigma} \in(0,1)$ and $\Theta=\frac{1}{\bar{\sigma}}  W_1^T  \Gamma W_1 +\frac{1}{1-\bar{\sigma}} W_2^T  \Gamma W_2$.
		\end{lemma}
		Our novel weighted reciprocally convex inequality is introduced in the following lemma.
		\begin{lemma}\label{WRCI}
			(Weighted reciprocally convex inequality) For matrices $\Gamma\in \mathbb{S}_+^{n} $, any matrix $S\in\mathbb{R}^{n \times n}$, $\vartheta_1\leq\vartheta(t)\leq\vartheta_2$ and $\left[\begin{array}{ll} \Gamma & S \\ *&   \Gamma\end{array}\right]\geq 0$, the following inequality holds,
			\begin{align}  \int_{t-\vartheta_2}^{t-\vartheta_1} e^{\delta\left(\vartheta_1+s-t\right)} \dot{\varrho}^{T}(s)  \Gamma \dot{\varrho}(s) d s  \geqslant \frac{\delta}{e^{\delta\left(\vartheta_2-\vartheta_1\right)}-1}\left[\begin{array}{l}\varUpsilon_1 \\ \varUpsilon_2\end{array}\right]^{T}\left[\begin{array}{ll} \Gamma & S \\ *&   \Gamma\end{array}\right]\left[\begin{array}{l}\varUpsilon_1 \\ \varUpsilon_2\end{array}\right]\end{align}
			where\\
			$\varUpsilon_1=\varrho(t-\vartheta(t))-\varrho(t-\vartheta_2),\quad\varUpsilon_2=\varrho(t-\vartheta_1)-\varrho(t-\vartheta(t)) $
		\end{lemma}
		\begin{proof}
			By splitting $\vartheta_1$ and $\vartheta_2$, then can get
			\begin{align}\notag
				&\int_{t-\vartheta_2}^{t-\vartheta_1} e^{\delta\left(s+\vartheta_1-t\right)} \dot{\varrho}^T(s) R \dot{\varrho}(s) d s\\\notag
				&=e^{\delta\left(\vartheta_1-\vartheta(t)\right)} \int_{t-\vartheta_2}^{t-\vartheta(t)} e^{\delta(s+\vartheta(t)-t)} \dot{\varrho}^T(s) R \dot{\varrho}(s) d s+\int_{t-\vartheta(t)}^{t-\vartheta_1} e^{\delta\left(s-\vartheta_1-t\right)} \dot{\varrho}^T(s) R \dot{\varrho}(s) d s \\\notag
				& \geq \frac{\delta}{e^{\delta\left(\vartheta_2-\vartheta_1\right)}-e^{\delta\left(\vartheta(t)-\vartheta_1\right)}}\varUpsilon_1^T R\varUpsilon_1 +\frac{\delta}{e^{\delta\left(\vartheta(t)-\vartheta_1\right)}-1}\varUpsilon_2^T R\varUpsilon_2 \\\notag
				& =\frac{\delta}{e^{\delta\left(\vartheta_2-\vartheta_1\right)}-1}\left\{\left[1+\frac{e^{\delta(\vartheta(t)-\vartheta_1)}-1}{e^{\delta\left(\vartheta_2-\vartheta_1\right)}-e^{\delta\left(\vartheta_1(t)-\vartheta_1\right)}}\right] \varUpsilon_1 R \varUpsilon_1\right. \\\notag
				& \left.+\left[1+\frac{e^{\delta\left(\vartheta_2-\vartheta_1\right)}-e^{\delta\left(\vartheta(t)-\vartheta_1\right)}}{e^{\delta\left(\vartheta(t)-\vartheta_1\right)}-1}\right] \varUpsilon_2 R \varUpsilon_2\right\}\\\notag
				& \geq\frac{\delta}{e^{\delta\left(\vartheta_2-\vartheta_1\right)}-1}\left[\begin{array}{l}
					\varUpsilon_1 \\
					\varUpsilon_2
				\end{array}\right]^T\left[\begin{array}{ll}
					R & S \\
					* & R
				\end{array}\right]\left[\begin{array}{l}
					\varUpsilon_1 \\
					\varUpsilon_2 
				\end{array}\right].\notag
		\end{align}\end{proof}
		\begin{remark}\label{R9}
		Reciprocally convex inequalities can be employed to analyze the stability of time-varying delay systems. By establishing appropriate LKFs and utilizing these inequalities, we can derive improved criteria for the stability of time-varying delay systems. As mentioned in the previous lemma, we introduce a new weighted reciprocally convex inequality that allows for the direct application of these inequalities to exponential stability. This serves as our main lemma and plays a crucial role in deriving the improved stability criteria in the following section.
		\end{remark}	
			
			\section{ Exponential stability analysis of NNs system}
			
		In this section, we demonstrate our primary result regarding the exponential stability of system \eqref{a6}.
			
			\begin{theorem} \label{Theorem1}
			For given scalars $ h > 0$ and $ \mu > 0$, system (\ref{a6}) achieves global exponential stability with a convergence rate \( k \) satisfying \( 0 < k < \min_{1 \leq i \leq n} c_i \). The stability is guaranteed by matrices   \(\mathscr{P}  \in \mathbb{S}_+^{3n }\), \( \mathscr{Q} \in \mathbb{S}_+^{2n } \), \( \mathscr{U}_i, \mathcal{Z}_i \in \mathbb{S}_+^{n } \) (\( i = 1, 2, 3, 4 \)), \( \mathscr{N}_j, \mathscr{M}_j \in \mathbb{S}_+^{n} \) (\( j = 1, 2 \)),  matrices \( \mathscr{D}_i = \text{diag}\{d_{i1}, \ldots, d_{in}\} \in \mathbb{D}_+^{n} \) and \( \mathscr{R}_i \in \mathbb{D}_+^{n } \) (\( i = 1, 2 \)), with any matrices \( \mathcal{S}_1 \in \mathbb{R}^{4n \times 4n} \) and \( \mathcal{S}_2 \in \mathbb{R}^{n \times n} \) that satisfy the following LMIs:
			$$\varSigma+\varTheta_1<0,\quad \varSigma+\varTheta_2<0, \quad \varGamma>0\quad \varOmega_1>0$$
			where\\
			\begin{align}\notag
					\varSigma&=\varDelta_1+\varDelta_2+\varDelta_3+\varDelta_4+\varDelta_5+\varPsi_1+\varPsi_2+\varPi, \quad
				\varTheta_1=\varSigma_1+\varSigma_2,\quad \varTheta_2=\psi_1+\psi_2, \\\nonumber \upsilon_i&=\biggl[\underbrace{0,0,\ldots,\overbrace{I}^i,\ldots,0}_{15}\biggl]^{T}_{15n\times n}, i= 1,2,\ldots,15,
				{	\upsilon_v=[-C,0_{n\times 2n},A,B,0_{n\times 10n}]^T,}\\\nonumber
					\mathscr{P}&=\left[\begin{array}{ccc}\mathscr{P}_{11} & \mathscr{P}_{12} & \mathscr{P}_{13} \\
					\ast & \mathscr{P}_{22} & \mathscr{P}_{23} \\
					\ast & \ast & \mathscr{P}_{33}\end{array} \right],
				\varGamma=\left[\begin{array}{ccc}\mathcal{Z}_{11} & \mathcal{S}_1 \\
					\mathcal{S}_1^T & \mathcal{Z}_{12}\end{array}\right],
				\varOmega=\left[\begin{array}{ccc}\mathcal{Z}_{13} & \mathcal{S}_1\\
					\mathcal{S}^T & \mathcal{Z}_{13}\end{array}\right], \varOmega_1=\left[\begin{array}{ccc}\mathcal{Z}_{4} & \mathcal{S}_2\\
					\mathcal{S}_2^T & \mathcal{Z}_{4}\end{array}\right],\\\nonumber
				\mathcal{Z}_{11}&=diag\{\mathcal{Z}_1+\mathscr{N}_1, 3(\mathcal{Z}_1+\mathscr{N}_1), 5(\mathcal{Z}_1+\mathscr{N}_1),7(\mathcal{Z}_1+\mathscr{N}_1)\},\\\nonumber
				\mathcal{Z}_{12}&=diag\{\mathcal{Z}_1+\mathscr{N}_2, 3(\mathcal{Z}_1+\mathscr{N}_2), 5(\mathcal{Z}_1+\mathscr{N}_2),7(\mathcal{Z}_1+\mathscr{N}_2)\},\,\mathcal{Z}_{13}=diag\{\mathcal{Z}_1,3\mathcal{Z}_1,5\mathcal{Z}_1,7\mathcal{Z}_1\}\\\nonumber	
			\mathscr{N}_{14}&=diag\{\mathscr{N}_1,3\mathscr{N}_1,5\mathscr{N}_1,7\mathscr{N}_1\},\quad\mathscr{N}_{15}=diag\{\mathscr{N}_2,3\mathscr{N}_2,5\mathscr{N}_2,7\mathscr{N}_2\},\\\nonumber
				\varGamma(i)&=[(\upsilon_i-\upsilon_{i+1}),\,(\upsilon_i+\upsilon_{i+1}-2\upsilon_{i+6}),\,(\upsilon_i-\upsilon_{i+1}+6\upsilon_{i+6}-6\upsilon_{i+9}),\\\notag
				&~~~~~(\upsilon_i+\upsilon_{i+1}-12\upsilon_{i+6}+30\upsilon_{i+9}-20\upsilon_{i+12})],\quad r(i)=[\upsilon_i-\upsilon_{i+1}],\quad i=1,2\\\nonumber	\varGamma(3)&=[(\upsilon_1-\upsilon_3),\,((h+c_1)\upsilon_1-c_1\upsilon_3-h\upsilon_6),\,((h^2+c_2h+c_3)\upsilon_1-c_3\upsilon_3-c_2h\upsilon_6-h^2\upsilon_{9})],\\\nonumber	\gamma_1&=[\varGamma(1),\,\varGamma(2)],\quad \gamma_2=[r(1),\,r(2)],\quad \hat{\zeta}=[\upsilon_1,\,h\upsilon_6,\,h\upsilon_{9}],\\\nonumber
				\varDelta_1&=sym\{k\hat{\zeta}\mathscr{P}\hat{\zeta}^T\,+\,2k[\upsilon_4\mathscr{D}_1\upsilon_1^T\,+\,(\upsilon_1\mathscr{L}\,-\,\upsilon_4)\mathscr{D}_2\upsilon_1^T]\,+\,\upsilon_4\mathscr{D}_1\upsilon_v^T+\,(\upsilon_1\mathscr{L}-\upsilon_4)\mathscr{D}_2\upsilon_v^T\},\\\nonumber
				\varDelta_2&=e^{2kh}\{[\upsilon_1,\,\upsilon_4]\mathscr{Q}[\upsilon_1,\,\upsilon_4]^T\,+\,\upsilon_1\mathscr{U}_1\upsilon_1^T\,+\,\upsilon_1\mathscr{U}_2\upsilon_1^T\}\,-\,(1\,-\,\mu) [\upsilon_2,\,\upsilon_5]\mathscr{Q}[\upsilon_2,\,\upsilon_5]^T\\\notag
				&~~~-e^{2k(h-\xi)}[\upsilon_{15}\mathscr{U}_2\upsilon_{15}^T-\upsilon_{15}\mathscr{U}_3\upsilon_{15}^T]\,-\,[\upsilon_3\mathscr{U}_1\upsilon_3^T\,+\,\upsilon_3\mathscr{U}_3\upsilon_3^T],\\\nonumber
				\chi_0&=\upsilon_1-\upsilon_3,\quad \chi_1=\bar{k}\upsilon_1+(h-\bar{k})\upsilon_3-h\upsilon_6,\\\notag
				\chi_2&=m\upsilon_1+(ch-m-h^2)\upsilon_3+(2h-c)h\upsilon_6-h^2\upsilon_9,\\\notag
				\chi_3&=r\upsilon_1+(h^3-\hbar h^2+qh-r)\upsilon_3-(3h^2-2\hbar h+q)h\upsilon_6-(2\hbar-6h)h^2\upsilon_9-h^3\upsilon_{12},\\\notag
				\varDelta_3&=h^2(\upsilon_v\mathcal{Z}_1\upsilon_v^T\,+\,\upsilon_1\mathcal{Z}_2\upsilon_1^T\,+\,\upsilon_v\mathcal{Z}_3\upsilon_v^T+\upsilon_v\mathcal{Z}_4\upsilon_v^T)\,-\,\biggr[\frac{h}{\langle g_0,g_0\rangle_w}h\upsilon_6\mathcal{Z}_2h\upsilon_6^T\\\notag
				&+\,\frac{h}{\langle g_1,g_1\rangle_w}((-h+\bar{k}){h}\upsilon_6\,+\frac{h^2}{2}\,\upsilon_{9})\mathcal{Z}_2((-h+\bar{k}){h}\upsilon_6\,+\frac{h^2}{2}\,\upsilon_{9})^T\,\\\nonumber
					&+\,\frac{h}{\langle g_2,g_2\rangle_w}((h^2-ch+m){h}\upsilon_6\,+(-2h+c)\frac{h^2}{2}\,\upsilon_{9}+\frac{h^3}{3}\upsilon_{12})\\\notag
					&\times \mathcal{Z}_2((h^2-ch+m){h}\upsilon_6\,+(-2h+c)\frac{h^2}{2}\,\upsilon_{9}+\frac{h^3}{3}\upsilon_{12})^T+\frac{h}{\left\langle g_0, g_0\right\rangle_w} \chi_0{ }^T \mathcal{Z}_3 \chi_0\,\\\nonumber
				&+\frac{h}{\left\langle g_1, g_1\right\rangle _w} \chi_1^{{T}} \mathcal{Z}_3 \chi_1+\frac{h}{\left\langle g_2, g_2\right\rangle_w} \chi_2^{{T}} \mathcal{Z}_3 \chi_2+\frac{h}{\left\langle g_3, g_3\right\rangle_w} \chi_3^{{T}} \mathcal{Z}_3 \chi_3\biggr],\\\nonumber	\varDelta_4&=\frac{h^2}{2}\upsilon_v\mathscr{N}_1\upsilon_v^T\,+\,\frac{h^2}{2}\upsilon_v\mathscr{N}_2\upsilon_v^T\,-\,e^{-2kh}\biggr[2(\upsilon_1-\upsilon_7)\mathscr{N}_1(\upsilon_1-\upsilon_7)^T\\\notag
				&+\,4(\upsilon_1\,+\,2\upsilon_7\,-\,3\upsilon_{10})\mathscr{N}_1(\upsilon_1\,+\,2\upsilon_7\,-\,3\upsilon_{10})^T\,+\,2(\upsilon_2-\upsilon_8)\mathscr{N}_1(\upsilon_2-\upsilon_8)^T\\\notag
				&+\,4(\upsilon_2\,+\,2\upsilon_8\,-\,3\upsilon_{11})\mathscr{N}_1(\upsilon_2\,+\,2\upsilon_8\,-\,3\upsilon_{11})^T\,+\,2(\upsilon_2-\upsilon_7)\mathscr{N}_2(\upsilon_2-\upsilon_7)^T\\\notag
				&+\,4(\upsilon_2\,-\,4\upsilon_7\,+\,3\upsilon_{10})\mathscr{N}_2(\upsilon_2\,-\,4\upsilon_7\,+\,3\upsilon_{10})^T\,+\,2(\upsilon_3-\upsilon_8)\mathscr{N}_2(\upsilon_3-\upsilon_8)^T\\\nonumber	+&\,4(\upsilon_3\,-\,4\upsilon_8\,+\,3\upsilon_{11})\mathscr{N}_2(\upsilon_3\,-\,4\upsilon_8\,+\,3\upsilon_{11})^T\biggr],\\\nonumber	\varDelta_5&=\frac{\mu}{h}\upsilon_1(\mathscr{M}_1{-\mathscr{M}_2})\upsilon_1^T,\quad	\varPsi_1=-e^{-2kh}\gamma_1\varOmega\gamma_1^T,\quad \varPsi_2= -\frac{2kh}{e^{2 k h}-1}\gamma_2\varOmega_1\gamma_2^T,\\\nonumber
				\Pi&=sym(\upsilon_1\mathscr{L}\mathscr{R}_1\upsilon_4^T\,-\,\upsilon_4\mathscr{R}_1\upsilon_4^T\,+\,\upsilon_2\mathscr{L}\mathscr{R}_2\upsilon_5^T\,-\,\upsilon_5\mathscr{R}_2\upsilon_5^T),\\\nonumber
				\varSigma_1&=sym([\upsilon_1,\,h\upsilon_7,\,h\upsilon_{9}]\mathscr{P}[\upsilon_v,\,\upsilon_1-\upsilon_3,\,2(\upsilon_1-\upsilon_6)]^T),\quad\varSigma_2=sym(k\upsilon_1\mathscr{M}_1\upsilon_1^T\,+\,\upsilon_1\mathscr{M}_1\upsilon_v^T),\\\nonumber
				\psi_1&=sym([\upsilon_1,\,h\upsilon_8,\,h\upsilon_{9}]\mathscr{P}[\upsilon_v,\,\upsilon_1-\upsilon_3,\,2(\upsilon_1-\upsilon_6)]^T),\quad\psi_2=sym(k\upsilon_1\mathscr{M}_2\upsilon_1^T\,+\,\upsilon_1\mathscr{M}_2\upsilon_v^T),\\\nonumber
				\delta&=\frac{h(t)}{h},\quad\tau=\frac{h-h(t)}{h},\quad \mathscr{L}=diag\{\mathscr{L}_1,\ldots,\mathscr{L}_n\}.
			\end{align}
			\end{theorem}				
							
			\begin{proof}
				Consider an augmented LKF as follow:
				$$\mathscr{V}(r(t))=\sum_{i=1}^5\mathscr{V}_i(r(t))$$
				where
				\begin{equation}
	\begin{array}{l}
	\mathscr{V}_1(r(t))=e^{2kt}\delta^T(t)\mathscr{P}\delta(t)+2\sum^n_{i=1}e^{2kt}d_{1i}\int_{0}^{z_i}f_i(s)ds
	+2\sum^n_{i=1}e^{2kt}d_{2i}\int_{0}^{z_i}(L_is-f_i(s))ds,\nonumber\\
	\mathscr{V}_2(r(t))=e^{2kh}\biggr\{\int_{t-h(t)}^{t}e^{2ks}\varepsilon^T(s)\mathscr{Q}\varepsilon(s)ds+\int_{t-h}^{t}e^{2ks}r^T(s)\mathscr{U}_1r(s)ds+
	\int_{t-\xi}^{t}e^{2ks}r^T(s)\mathscr{U}_2r(s)ds\nonumber\\
	~~~~~~~~~~~~~~+\int_{t-h}^{t-\xi}e^{2ks}r^T(s)\mathscr{U}_3r(s)ds\biggr\},\nonumber\\
	\mathscr{V}_3(r(t))=h\biggr\{\int_{-h}^{0}\int_{t+u}^{t}e^{2ks}\dot{r}(s)^T\mathcal{Z}_1\dot{r}(s)dsdu+\int_{-h}^{0}\int_{t+u}^{t}e^{2ks}r(s)^T\mathcal{Z}_2r(s)dsdu\nonumber\\
	~~~~~~~~~~~~~~+\int_{-h}^{0}\int_{t+u}^{t}e^{2ks}\dot{r}(s)^T\mathcal{Z}_3\dot{r}(s)dsdu+\int_{-h}^{0}\int_{t+u}^{t}e^{2ks}\dot{r}(s)^T\mathcal{Z}_4\dot{r}(s)dsdu\biggr\},\nonumber\\
	\mathscr{V}_4(r(t))= \int_{-h}^{0}\int_{v}^{0}\int_{t+u}^{t}e^{2ks}\dot{r}^T(s)\mathscr{N}_1\dot{r}(s)dsdudv+ \int_{-h}^{0}\int_{-h}^{v}\int_{t+u}^{t}e^{2ks}\dot{r}^T(s)\mathscr{N}_2\dot{r}(s)dsdudv,\nonumber\\
	\mathscr{V}_5(r(t))= \frac{h(t)}{h}e^{2kt}r^T(t)\mathscr{M}_1r(t)+\frac{h-h(t)}{h}e^{2kt}r^T(t)\mathscr{M}_2r(t).\nonumber
\end{array}
\end{equation}				
				Let
				$\chi^T(t) = [r^T(t),\,r^T(t-h(t)),\,r^T(t-h),\,g^T(r(t)),\,g^T(r(t-h(t)),
				\ \frac{1}{h}\int_{t-h}^{t}r^T(s)ds,\\\,\frac{1}{h(t)}\int_{t-h(t)}^{t}r^T(s)ds,\,\frac{1}{h-h(t)}\int_{t-h}^{t-h(t)}r^T(s)ds,\,
				\frac{2}{h^2}\int_{-h}^{0}\int_{t+u}^{t}r^T(s)dsdu,\,\frac{2}{h^2(t)}\int_{-h(t)}^{0}\int_{t+u}^{t}r^T(s)dsdu,\,\nonumber\\
				\frac{2}{(h-h(t))^2}\int_{-h}^{-h(t)}\int_{t+u}^{t-h(t)}r^T(s)dsdu,\frac{6}{h^3}\int_{-h}^{0}\int_{u}^{0}\int_{t+s}^{t}r^T(v)dvdsdu,\frac{6}{h(t)^3}\int_{-h(t)}^{0}\int_{u}^{0}\int_{t+s}^{t}r^T(v)dvdsdu,\nonumber\\
				\frac{6}{(h-h(t))^3}\int_{-h}^{-h(t)}\int_{u}^{-h(t)}\int_{t+s}^{t}r^T(v)dvdsdu,r^T(t-\xi) ],$ \\$
				\delta^T(t)  =  [r^T(t),\,\int_{t-h}^{t}r^T(s)ds,\,\frac{2}{h}\int_{-h}^{0}\int_{t+u}^{t}r^T(s)dsdu],$ $
				\varepsilon^T(t)  = [r^T(t),\,g^T(r(t))].
				$\\
					First, we will estimate the time derivative of \( V_i(r(t)) \) along the trajectories of (\ref{a6}). Particular noted when \( i = 1, 2, 5 \), these three estimates provided are similar to those in \cite{SCY}, but we will present some key steps to ensure the completeness of our presentation.
			\begin{align}
			\dot{V}_1(r(t))\leq&
			e^{2kt}\chi^T(t)[\varDelta_1+\delta\varSigma_1+\tau\psi_1]\chi(t) \nonumber\\
			\dot{V}_2(r(t))  =&  e^{2kh}\bigg[e^{2kt}\epsilon^T(t)\mathscr{Q}\epsilon(t)-e^{2k(t-h(t))}(1-{\dot{h}(t)})\epsilon^T(t-h(t)\mathscr{Q}\epsilon(t-h(t))
			+e^{2kt}r^T(t)\mathscr{U}_1r(t)\nonumber\\ &-e^{2k(t-h)}r^T(t-h)\mathscr{U}_1r(t-h)+e^{2kt}r^T(t)\mathscr{U}_2r(t)
			-e^{2k(t-\xi)}r^T(t-\xi)\mathscr{U}_2r(t-\xi)\nonumber\\ &+e^{2k(t-\xi)}r^T(t-\xi)\mathscr{U}_3r(t-\xi)
			-e^{2k(t-h)}r^T(t-h)\mathscr{U}_3r(t-h)\bigg]
			\nonumber\\  \leq  &e^{2kt}\chi^T(t)\bigg\{e^{2kh}[\upsilon_1,\upsilon_4]\mathscr{Q}[\upsilon_1,\upsilon_4]^T-(1-\mu)[\upsilon_2,\upsilon_5]\mathscr{Q}[\upsilon_2,\upsilon_5]^T
			+e^{2kh}\upsilon_1\mathscr{U}_1\upsilon_1^T-\upsilon_3U_1\upsilon_3^T\nonumber\\ &+e^{2kh}\upsilon_1\mathscr{U}_2\upsilon_1^T-e^{2k(h-\xi)}\upsilon_{15}\mathscr{U}_2\upsilon_{15}^T
			+e^{2k(h-\xi)}\upsilon_{15}\mathscr{U}_3\upsilon_{15}^T-\upsilon_3\mathscr{U}_3\upsilon_3^T\bigg\}\chi(t)
			\nonumber\\
			=&  e^{2kt}\chi^T(t)\varDelta_2\chi(t) \nonumber\\
				\dot{V}_5(r(t)) = & e^{2kt}\delta\big[2kr^T(t)\mathscr{M}_1r(t)+2r^T(t)\mathscr{M}_1\dot{r}(t)\big]+\frac{\dot{h}(t)}{h}e^{2kt}r^T(t)\mathscr{M}_1r(t)\nonumber\\
				+&e^{2kt}\tau\big[2kr^T(t)\mathscr{M}_2r(t)+2r^T(t)\mathscr{M}_2\dot{r}(t)\big]-\frac{\dot{h}(t)}{h}e^{2kt}r^T(t)\mathscr{M}_2r(t)\nonumber\\
				 \leq & e^{2kt}\chi^T(t)\big\{\delta sym(k\upsilon_1\mathscr{M}_1\upsilon_1^T+\upsilon_1\mathscr{M}_1\upsilon_v^T)+\tau sym(k\upsilon_1\mathscr{M}_2\upsilon_1^T+\upsilon_1\mathscr{M}_2\upsilon_v^T)+\frac{\mu}{h}\upsilon_1(\mathscr{M}_1{-\mathscr{M}_2})\upsilon_1^T\big\}\chi(t)\nonumber\\
			 = &  e^{2kt}\chi^T(t)\{\varDelta_5+\delta\varSigma_2+\tau\psi_2\}\chi(t).\nonumber
				\end{align}
			We apply the innovative inequalities presented in Lemma \ref{L2}, Corollary \ref{C4} and Lemma \ref{WRCI} to estimate \( \dot{V}_3 \) and  \( \dot{V}_4 \). Accordingly, we express:
				\begin{equation}
					\begin{aligned}
						\dot{V}_3(r(t))  =&  e^{2kt}\bigg[h^2\dot{r}^T(t)(\mathcal{Z}_1+\mathcal{Z}_3+\mathcal{Z}_4)\dot{r}(t)+h^2r^T(t)Z_2r(t)
						-h\int_{t-h}^{t}e^{2k(s-t)}r^T(s)\mathcal{Z}_2r(s)ds\nonumber\\ &-h\int_{t-h}^{t}e^{2k(s-t)}\dot{r}^T(s)\mathcal{Z}_3\dot{r}(s)ds
						-h\int_{t-h}^{t}e^{2k(s-t)}\dot{r}^T(s)\mathcal{Z}_1\dot{r}(s)ds\nonumber\\
						&-h\int_{t-h}^{t}e^{2k(s-t)}\dot{r}^T(s)\mathcal{Z}_4\dot{r}(s)ds\bigg]
					\end{aligned}
				\end{equation}
			 By simple calculation and bounded by inequality (\ref{e10}), we can get the following inequality. 
			\begin{align}
				&-h\int_{t-h}^{t}e^{2k(s-t)}\dot{r}^T(s)\mathcal{Z}_1\dot{r}(s)ds\notag\\
				& =-h \int_{t-h(t)}^t e^{2 k(s-t)} \dot{r}^{T}(s) \mathcal{Z}_1 \dot{r}(s) d s-h \int_{t-h}^{t-h(t)} e^{2 k(s-t)} \dot{r}^{T}(s) \mathcal{Z}_1 \dot{r}(s) d s \notag\\
				& \leq-h e^{2 k(-h(t))} \int_{t-h(t)}^t r^{T}(s) \mathcal{Z}_1 r(s) d s-h e^{2 k(-h)} \int_{t-h}^{t-h(t)} \dot{r}^{T}(s) \mathcal{Z}_1 r(s) d s\notag \\
				& \leq-h e^{-2 k h} \bigg(\int_{t-h(t)}^t r^{T}(s) \mathcal{Z}_1 \dot{r}(s) d s+ \int_{t-h}^{t-h(t)} r^{T}(s) \mathcal{Z}_1 \dot{r}(s) d s\bigg)\notag \\
				& \leq e^{-2kh}\chi^T(t)\bigg\{\frac{1}{\delta}\varGamma(1)\mathcal{Z}_{13}\varGamma^T(1)+\frac{1}{\tau}\varGamma(2)\mathcal{Z}_{13}\varGamma^T(2)\bigg\}\chi(t).
			\end{align}		

		 Utilizing Lemma \ref{C4} and focusing on  $\hat{m}=2$, the following results can be obtained. 
			{\footnotesize 
					\begin{align*}
						-h\int_{t-h}^{t}&e^{2k(s-t)}r^T(s)\mathcal{Z}_2r(s)ds \nonumber\\
						 \leq &-\frac{h}{\langle g_0,g_0\rangle_w}\left(\int_{t-h}^t r(s) d s\right)^T \mathcal{Z}_2\left(\int_{t-h}^t r(s) d s\right) \\
						 &-\frac{h}{\left\langle g_1, g_1\right\rangle_w}\left(\int_{t-h}^t(s+\bar{k}) r(s) d s\right)^T \mathcal{Z}_2\left(\int_{t-h}^t(s+\bar{k}) r(s) d s\right) \\
						& -\frac{h}{\left\langle g_2, g_2\right\rangle_w}\left(\int_{t-h}^t\left(s^2+c s+m\right) r(s) d s\right)^T \mathcal{Z}_2\left(\int_{t-h}^t\left(s^2+c s+m\right) r(s) d s\right) \\
						=&-\frac{h}{\langle g_0,g_0\rangle_w}\left(\int_{t-h}^t r(s) d s\right)^T \mathcal{Z}_2\left(\int_{t-h}^t r(s) d s\right) \\
						&-\frac{h}{\left\langle g_1, g_1\right\rangle_w}\left((-h+\bar{k})\int_{t-h}^t r(s) d s+\int_{0}^t\int_{t-h}^tr(s)dsdu\right)^T \mathcal{Z}_2\left((-h+\bar{k})\int_{t-h}^t r(s) d s+\int_{0}^t\int_{t-h}^tr(s)dsdu\right) \\
						&-\frac{h}{\left\langle g_2, g_2\right\rangle_w}\left((h^2-ch+m)\int_{t-h}^t r(s) d s+(-2h+c)\int_{0}^t\int_{t-h}^tr(s)dsdu+2\int_{0}^t\int_{0}^s\int_{t+u}^tr(v)dvduds\right)^T \\
						&\times \mathcal{Z}_2\left((h^2-ch+m)\int_{t-h}^t r(s) d s+(-2h+c)\int_{0}^t\int_{t-h}^tr(s)dsdu+2\int_{0}^t\int_{0}^s\int_{t+u}^tr(v)dvduds\right)\\
					\end{align*}}
				For this term, by Corollary \ref{C4} and take $\hat{m}$=3
				{\footnotesize 
					\begin{align*}
						-h\int_{t-h}^{t}&e^{2k(s-t)}\dot{r}^T(s)\mathcal{Z}_3\dot{r}(s)ds\nonumber\\
						& \leqslant-\frac{h}{\left\langle g_0, g_0\right\rangle_w}\left(r(t)-r(t-h)^{T} \mathcal{Z}_3(r(t)-r(t-h)\right) \\
						& -\frac{h}{\left\langle g_1, g_1\right\rangle_\omega}\left(\bar{k} r(t)-(\bar{k}-h) r(t-h)-\int_{t-h}^t r(s) d s\right)^{T} \mathcal{Z}_3\left(\bar{k} r(t)-(\bar{k}-h) r(t-h)-\int_{t-h}^t r(s) d s\right) \\
						&-\frac{h}{\left\langle g_2, g_2\right\rangle_w} \left(m r(t)-\left(h^2-c h+m\right) r(t-h)-(c-2 h) \int_{t-h}^t r(s) d s-2 \int_{-h}^0 \int_{t+u}^t r(u) d u d s\right)^T\\
						&\times \mathcal{Z}_3\left(m r(t)-\left(h^2-c h+m\right) r(t-h)-(c-2 h) \int_{t-h}^t r(s) d s-2 \int_{-h}^0 \int_{t+u}^t r(u) d u d s\right)\\
						& -\frac{h}{\left\langle g_3, g_3\right\rangle_\omega}\left(r r(t)-\left(-h^3+\hbar h^2-q h+r\right) r(t-h)-\left(3 h^2-2 h \hbar+q\right)\int_{t-h}^{t}r(s) d s \right.\\
						&\quad \quad\quad\quad\quad\quad \left.-(2 \hbar-6 h) \int_{-h}^0 \int_{1+u}^t r(u) d u d s-6 \int_{-h}^0 \int_s^0 \int_{t+u}^t r(v)d v d u d s\right)^{T}\\
						&\times \mathcal{Z}_3\left(r r(t)-\left(-h^3+\hbar h^2-q h+r\right) r(t-h)-\left(3 h^2-2 h \hbar+q\right)\int_{t-h}^{t}r(s) d s \right.\\
						&\quad \quad\quad\quad\quad\quad \left.-(2 \hbar-6 h) \int_{-h}^0 \int_{1+u}^t r(u) d u d s-6 \int_{-h}^0 \int_s^0 \int_{t+u}^t r(v)d v d u d s\right)
					\end{align*}}
				For the last term of $\dot{V}_3(r(t))$, by out novel weighted reciprocally convex inequality Lemma \ref{WRCI}
				\begin{align*}
					&-h \int_{t-h}^t e^{2 k(s-t)} \dot{r}(s) \mathcal{Z}_4 \dot{r}(s) d s \\
						& =-h e^{2 k(-h(t))} \int_{t-h}^{t-h(t)} e^{2 k(s+h(t)-t)} \dot{r}(s) \mathcal{Z}_4 \dot{r}(s) d s  -h \int_{t-h(t)}^t e^{2 k(s-t)} \dot{r}^{T}(s) \mathcal{Z}_4 \dot{r}(s) d s \\
						& \leq-\frac{ 2 kh}{e^{2 k h}-e^{2 k h(t)}}(r(t-h(t))-r(t-h))^{T} \mathcal{Z}_4(r(t-h(t)-r(t-h)) \\
						& -\frac{ 2 kh}{e^{2 k h(t)}-1}(r(t)-r(t-h(t))^{T} \mathcal{Z}_4(r(t)-r(t-h(t)) \\
						& =-\frac{  2 kh}{e^{2 k h}-1}\left\{\left[1+\frac{e^{2 kh(t)}-1}{e^{2 k h}-e^{2 k h(t)}}\right](r(t-h(t)-r(t-h)))^T\mathcal{Z}_4(r(t-h(t)-r(t-h))) \right.\\						
						& \left.+\left[1+\frac{e^{2 k h}-e^{2 k h(t)}}{e^{2 k h(t)}-1}\right](r(t)-r(t-h(t)))^T \mathcal{Z}_4(r(t)-r(t-h(t)))\right\}\\
						&\leq-\frac{2kh}{e^{2 k h}-1}\left[\begin{array}{l}
							r_{(1)} \\
							r_{(2)}
						\end{array}\right]^{\top}\left[\begin{array}{ccc}
							 \mathcal{Z}_4 & \mathcal{S}_4 \\
						      * & \mathcal{Z}_4
						\end{array}\right]\left[\begin{array}{l}
							r_{(1)} \\
							r_{(2)}
						\end{array}\right] \\
						&=\chi^T(t)\varPsi_2\chi(t)
				\end{align*}
				Consequently
				\setlength\arraycolsep{2pt}
				\begin{eqnarray}\nonumber
					\dot{V}_3(r(t)) & \leq & e^{2kt}\chi^T(t)\Bigg\{\varDelta_3+\varPsi_2-e^{-2kh}\bigg[\frac{1}{\delta}\gamma(1)\mathcal{Z}_{13}\gamma^T(1)
					+\frac{1}{\tau}\gamma(2)\mathcal{Z}_{13}\gamma^T(2)\bigg]\Bigg\}\chi(t).\nonumber
				\end{eqnarray}
				For $\dot{V}_4(r(t)) $, using inequality (\ref*{e10}) to deal with last two terms,  other steps are same with \cite{LiuX}
			\begin{align}
				\dot{V}_4(r(t))  = & \frac{h^2}{2}e^{2kt}\dot{r}^T(t)(\mathscr{N}_1+\mathscr{N}_2)\dot{r}(t)-\int_{-h}^{0}\int_{t+u}^{t}e^{2ks}\dot{r}^T(s)\mathscr{N}_1\dot{r}(s)dsdu\nonumber\\
				& -\int_{-h}^{0}\int_{t-h}^{t+u}e^{2ks}\dot{r}^T(s)\mathscr{N}_2\dot{r}(s)dsdu\nonumber\\
				\leq & \frac{h^2}{2}e^{2kt}\dot{r}^T(t)(\mathscr{N}_1+\mathscr{N}_2)\dot{r}(t)-e^{2k(t-h)}\int_{-h}^{0}\int_{t+u}^{t}\dot{r}^T(s)\mathscr{N}_1\dot{r}(s)dsdu\nonumber\\
				& -e^{2k(t-h)}\int_{-h}^{0}\int_{t-h}^{t+u}\dot{r}^T(s)\mathscr{N}_2\dot{r}(s)dsdu\nonumber\\
				\leq& e^{2kt}\chi^T(t)\Bigg\{\varDelta_4-e^{-2kh}\bigg[\left(\frac{1}{\delta}-1\right) h(t) \int_{t-h(t)}^t z^T(s) N_1 \dot{z}(s) d s\notag\\
				&+\left(\frac{1}{\tau}-1\right)(h-h(t)) \int_{t-h}^{t-h(t)} z^T(s) N_2 z(s) d s\bigg]\Bigg\}\chi(t)\notag	\end{align}
			The last two terms can be bounded by (\ref{e10}) as follow 
		\begin{align}
				\dot{V}_4(r(t))\leq & e^{2kt}\chi^T(t)\Bigg\{\varDelta_4-e^{-2kh}\bigg[\Big(\frac1\delta-1\Big)\varGamma(1)\mathscr{N}_{14}\varGamma^T(1)+\Big(\frac1\tau-1\Big)\varGamma(2)\mathscr{N}_{15}\varGamma^T(2)\bigg]\Bigg\}\chi(t)\nonumber.
			\end{align}
				By taking into account the assumptions in {\eqref{a7}} at \( r(t) \) and \( r(t-h(t)) \), for any positive diagonal matrices \( R_1 \) and \( R_2 \), we obtain:
				{\setlength\arraycolsep{2pt}
					\begin{eqnarray}\label{R-inequality}
						0 & \leq & 2e^{2kt}[r^T(t)L\mathscr{R}_1{f}(r(t))-f^T(r(t))\mathscr{R}_1f(r(t))
						\nonumber\\		& & +r^T(t-h(t)) \mathscr{L}\mathscr{R}_2{f}(r(t-h(t)))-f^T(r(t-h(t)))\mathscr{R}_2f(r(t-h(t)))]\nonumber\\
						& = & e^{2kt}\chi^T(t)\varPi\chi(t).
				\end{eqnarray}}
				By {lemma \ref{L7}}, we have
				{\setlength\arraycolsep{2pt}
					\begin{eqnarray}
						& -\chi^T(t)\Bigg\{\frac{1}{\delta}\varGamma(1)\mathcal{Z}_{1}\varGamma^T(1)+\frac{1}{\tau}\varGamma(2)\mathcal{Z}_{1}\varGamma^T(2) +\frac{1}{\delta}\varGamma(1)\mathscr{N}_{14}\varGamma^T(1)\nonumber\\
						&+\frac{1}{\tau}\varGamma(2)\mathscr{N}_{15}\varGamma^T(2) -\varGamma(1)\mathscr{N}_{14}\varGamma^T(1)-\varGamma(2)\mathscr{N}_{15}\varGamma^T(2)\Bigg\}\chi(t)\nonumber\\
						&\leq  \chi^T(t)\big\{-e^{-2kh}\varGamma\varOmega\varGamma^T\big\}\chi(t)=\chi^T(t)\varPsi_1\chi(t).\nonumber
				\end{eqnarray}}
				Therefore, we have 
				\[
				\dot{V}(r(t)) \leq e^{2kt} \chi^T(t) \{ \varSigma + \delta \varTheta_1 + \tau \varTheta_2 \} \chi(t).
				\]
				Given that \( \varSigma + \varTheta_1 < 0 \) and \( \varSigma + \varTheta_2 < 0 \), along with \( \delta + \tau = 1 \), it follows that \( \varSigma + \delta \varTheta_1 + \tau \varTheta_2 < 0 \). Thus, for any \( \chi(t) \neq 0 \), we conclude that \( \dot{V}(r(t)) < 0 \).
				Thus,
				$$V(r(0))\leq\Lambda\|\phi\|^2,$$
				and\\
				{\setlength\arraycolsep{2pt}
					\begin{eqnarray}
						\Lambda & = & \lambda_{max}(\mathscr{P})(1+2h^2)+2\lambda_{max}(\mathscr{D}_1\mathscr{L})+2\lambda_{max}(\mathscr{D}_2\mathscr{L})+he^{2kh}\lambda_{max}(\mathscr{Q})\nonumber\\
						& & \times[1+\lambda_{max}(\mathscr{L}^2)]+he^{2kh}(\lambda_{max}(\mathscr{U}_1)+\lambda_{max}(\mathscr{U}_2)+\lambda_{max}(\mathscr{U}_3))\nonumber\\
						& & +\Bigg[\frac{3h^3}{2}\lambda_{max}(\mathcal{Z}_1)+\frac{3h^3}{2}\lambda_{max}(\mathcal{Z}_3)+\frac{3h^3}{2}\lambda_{max}(\mathcal{Z}_4)+\frac{h^3}{6}\lambda_{max}(\mathscr{N}_1)+\frac{h^3}{2}\lambda_{max}(\mathscr{N}_2)\Bigg]\nonumber\\
						& &
						\times \big[\lambda_{max}(K_0^TK_0)+\lambda_{max}(K_1^TK_1)\lambda_{max}(\mathscr{L}^2)+\lambda_{max}(K_2^TK_2)\lambda_{max}(\mathscr{L}^2)\big]\nonumber\\
						& &
						+h\lambda_{max}(\mathscr{M}_1+\mathscr{M}_2)+\frac{h^3}{2}\lambda_{max}(\mathcal{Z}_2).\nonumber
				\end{eqnarray}}
			Simultaneously, we obtain:
				$$V(r(t))\geq e^{2kt}\delta^T(t)\mathscr{P}\delta(t)\geq e^{2kt}\lambda_{min}(\mathscr{P})\|\delta(t)\|^2\geq e^{2kt}\lambda_{min}(\mathscr{P})\|r(t)\|^2.$$
				Therefore,
				$$\|r(t)\|\leq\sqrt{\frac{\Lambda}{\lambda_{max}(\mathscr{P})}}\|\phi\|e^{-kt},$$
				{the proof is concluded.}
			\end{proof}
			%

			\bigskip
			\section{Numerical example}
		We now present two examples along with their simulations to show that advantages of the results obtained.\\
			{\bf Example 1}  {\rm\cite{WuM,JiM1,JiM2,LiuX,HY2023,SCY}} Consider the NNs time delayed system \eqref{a6} with:
			\setlength{\abovedisplayskip}{10pt}
			\begin{equation}\nonumber
			\begin{array}{l}
				K_1=\left[\begin{matrix}
					-1& 0.5 \\
					0.5& -1
				\end{matrix}\right],\ \
				K_2=\left[
				\begin{matrix}
					-0.5& 0.5 \\
					0.5 & 0.5
				\end{matrix}
				\right],\ \ K_0=\diag\{2,3.5\},
				\ \ \mathscr{L}_1=1,\ \ \mathscr{L}_2=1.
			\end{array}
			\end{equation}
			
			
			For different values of \( \mu \) and \( h = 1 \), the maximum allowable exponential convergence rate \( k \) for the system is documented in Table \ref{table1}. The table indicates that our criterion is more effective than those presented in \cite{WuM, JiM1, JiM2,LIUYY,LiuX,HY2023,zck2023,BINYY,SCY,2022}. The NoDVs generated by the new method are significantly smaller than those in \cite{LIUYY, HY2023, zck2023, BINYY, 2022}, while also reducing conservatism in the exponential stability of NNs systems \eqref{a6}.
			

			
			\begin{table}[H]
			\caption{Allowable values of $k$ for different $\mu$ and $h=1$ (Example 1).}\label{table1}
			\setlength{\tabcolsep}{7mm}
			\centering
			\begin{tabular}{lcccccc}
				\hline
				$\mu$    &  0.8   &0.9 & NoDVs  \\ \hline
				\cite{WuM} & 0.8643& 0.8344 &  $3n^2+12n$  \\
				\cite{JiM1}& 0.8696& 0.8354 & $13n^2+6n$  \\
				\cite{JiM2} & 0.8784& 0.8484& $7n^2+8n$ \\
				\cite{LIUYY}  &0.9341& 0.9057 & ${96n^2 + 10n}$ \\
				\cite{LiuX}  &0.9382& 0.9104 & ${20.5n^2 + 12.5n}$ \\
				\cite{BINYY}	Theorem 1  &0.9770 &0.9450 &$181.5n^2 + 11.5n$  \\  
				\cite{BINYY}	Theorem 2  &0.9283 &0.8989 &${141.5n^2 + 11.5n}$  \\  
				\cite{HY2023} &1.0024 &0.9696& ${189.5n^2 + 12.5n}$  \\
				\cite{zck2023} &1.0082 &0.9745& ${306n^2 + 9.5n}$  \\  
				\cite{SCY}  &1.0889 &1.0732&  ${20.5n^2 + 11.5n}$  \\ 
					\cite{2022} &1.0852 &1.0810& ${173n^2 + 16n}$  \\ 
					Theorem \ref{Theorem1}  &1.2549 &1.2300 &${29n^2 + 12n}$  \\  \hline
			\end{tabular}
			\end{table}
			
			{\bf Example 2}   Consider the NNs time-varying delayed system \eqref{a6} with the following matrices and compared with  {\cite{WuM,ZhengC,JiM1,JiM2,LIUYY,2020,LiuX}}:
			\setlength{\abovedisplayskip}{10pt}
			\begin{equation}\nonumber
			\begin{array}{l}
				K_1=\left[\begin{matrix}
					-0.0373& 0.4852& -0.3351&0.2336 \\
					-1.6033&0.5988&-0.3224&1.2352\\
					0.3394& -0.0860&-0.3824& -0.5785\\
					-0.1311&0.3253&-0.9534&-0.5015
				\end{matrix}\right],\ \
				K_2=\left[
				\begin{matrix}
					0.8674&-1.2405&-0.5325&-0.0220\\
					0.0474&-0.9164&0.0360&0.9816\\
					1.8495&2.6117&-0.3788&0.0824\\
					-2.0413&0.5179&1.1734&-0.2775
				\end{matrix}
				\right],\\
				
				\ \ K_0=\diag\{1.2769, 0.6231, 0.9230, 0.4480\},\\
				\ \ \mathscr{L}_1=0.1137,\ \ \mathscr{L}_2=0.1279,\ \ \mathscr{L}_3=0.7994,\ \ \mathscr{L}_4=0.2368.
			\end{array}
			\end{equation}
			
		In this example, following the approach in \cite{LiuX}, we compare our method with those proposed in \cite{WuM, ZhengC, JiM1, JiM2, LIUYY, 2020, LiuX} by setting \( k = 10^{-3} \). The maximal upper bounds of \( h(t) \) along with the corresponding NoDVs for different values of \( \mu \) are presented in Table \ref{table2}. The results demonstrate the enhancements achieved by our method.

			
			{Figure \ref{figure2} illustrates the trajectory of the delayed system \eqref{a6} with the initial condition $r(0)=[-1,-0.5,0.5,1]^T,$ $ h(t)=2.4+0.9sin(t),$ $f(r(t))=[0.1137tanh(r_1(t)),0.1279tanh(r_2(t)),\\ 0.7994tanh(r_3(t)),0.2368tanh(r_4(t))].$ }
			
			
			\begin{table}[H]
			\caption{Allowable $h$ for various $\mu$  (Example 2).}\label{table2}
			\setlength{\tabcolsep}{6mm}
			\centering
			\begin{tabular}{lcccccc}
				\hline
				$\mu$ &0.5   &  0.8   &0.9 & NoDVs  \\ \hline
				\cite{WuM} &2.5379 &2.1766& 2.0853& $3n^2+12n$  \\
				\cite{ZhengC}&2.6711& 2.2977 &2.1783& $4.5n^2 + 17.5n$  \\
				\cite{JiM1}&3.4311& 2.5710& 2.4147 & $13n^2+6n$  \\
				\cite{JiM2}&3.6954& 2.7711 &2.5795&  $7n^2+8n$ \\
				\cite{LIUYY}&-& 2.9911 &2.5300&  ${125.5n^2 + 10.5n}$ \\
				\cite{2020}&-& 3.0408 &2.6611&  ${99n^2 + 10n}$ \\
				\cite{LiuX}Theorem 3.1&3.8709& 3.3442 &3.1291&  ${20.5n^2 + 12.5n}$ \\
			
				Theorem of \ref{Theorem1} &4.0200 &3.6000 &3.3000 &${29n^2 + 12n}$  \\  \hline
			\end{tabular}
			
			\end{table}
			\begin{figure}[htb!]
				\setlength{\unitlength}{1cm} 
				\begin{center}
					\resizebox{!}{8cm}{\includegraphics{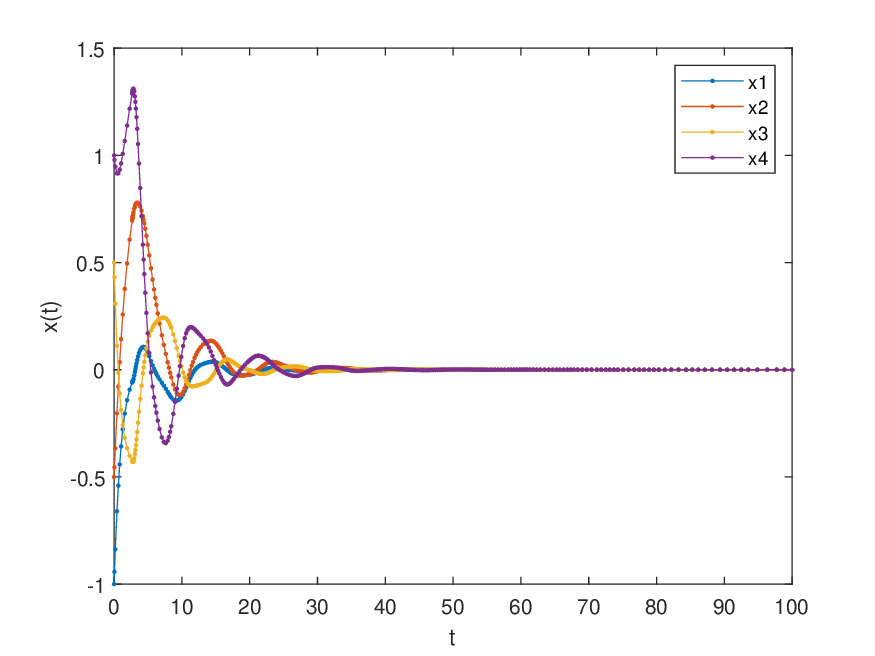}}
					\vspace{-0.1cm}
					\caption{  Trajectory  of Example 2.}\label{figure2}
				\end{center}
			\end{figure}

			\section{Conclusion}
			Using the new weighted reciprocally convex inequality and auxiliary function inequalities, we studied the exponential stability of neural network systems with time delays. By applying the improved LKF, using some new inequities to estimate the derivative terms of the LKF, we obtained enhanced stability criteria. This approach allows for ADUB to be achieved while ensure that number of decision variables increases are minimized or even reduced. The effectiveness of the new method is validated through numerical results.

			\end{document}